\documentclass[12pt,a4paper]{article}

\usepackage{latexsym, amssymb, amsthm}

\usepackage{amsmath}
\usepackage{euscript}

\makeatletter
\def\eqnarray{\stepcounter{equation}\let\@currentlabel=\theequation
\global\@eqnswtrue
\tabskip\@centering\let\\=\@eqncr
$$\halign to \displaywidth\bgroup\hfil\global\@eqcnt\z@
  $\displaystyle\tabskip\z@{##}$&\global\@eqcnt\@ne
  \hfil$\displaystyle{{}##{}}$\hfil
  &\global\@eqcnt\tw@ $\displaystyle{##}$\hfil
  \tabskip\@centering&\llap{##}\tabskip\z@\cr}

\def\endeqnarray{\@@eqncr\egroup
      \global\advance\c@equation\m@ne$$\global\@ignoretrue}

\def\@yeqncr{\@ifnextchar [{\@xeqncr}{\@xeqncr[5pt]}}
\makeatother

\allowdisplaybreaks

      \def\dC{{\mathbb C}}

   \def\dN{{\mathbb N}}

\def\Ci{{\mathbb C}}
\def\Ni{{\mathbb N}}
\def\Ri{{\mathbb R}}

\def\cA{{\EuScript A}}      
\def\cD{{\EuScript D}}      
   \def\cH{{\EuScript H}}   
      \def\cL{{\EuScript L}}
      \def\cO{{\EuScript O}}

\def\bm\chi{\mbox{\boldmath$\chi$}}

\def\ker{{\rm ker\,}}

\def\dom{{\rm dom\,}}

\def\Tr{{\rm Tr\,}}

\let\xker=\ker \def\ker{{\xker\,}}

\def\supp{{\rm supp\,}}

\def\senki{{\lbrack\negthinspace [\bot ]\negthinspace\rbrack}}
\def\senki+{{\lbrack\negthinspace [+] \negthinspace\rbrack}}

\newtheorem{theorem}{Theorem}[section]
\newtheorem{proposition}[theorem]{Proposition}
\newtheorem{corollary}[theorem]{Corollary}
\newtheorem{lemma}[theorem]{Lemma}
\theoremstyle{definition}

\newtheorem{remark}[theorem]{Remark}

\numberwithin{equation}{section}

\newcounter{teller}

\newenvironment{tabel}{\begin{list}%
{\rm  (\alph{teller})\hfill}{\usecounter{teller} \leftmargin=1.1cm
\labelwidth=1.1cm \labelsep=0cm \parsep=0cm}
                      }{\end{list}}

\newcounter{tellerr}

\newenvironment{tabeleq}{\begin{list}%
{\rm  (\roman{tellerr})\hfill}{\usecounter{tellerr} \leftmargin=1.1cm
\labelwidth=1.1cm \labelsep=0cm \parsep=0cm}
                         }{\end{list}}

\newcommand{\RRe}{\mathop{\rm Re}}

\DeclareMathAlphabet\gothic{U}{euf}{m}{n}
\newcommand{\gotb}{\gothic{b}}
\newcommand{\gota}{\gothic{a}}
\newcommand{\spann}{\mathop{\rm span}}
\newcommand{\divv}{\mathop{\rm div}}

\begin{document}

\thispagestyle{empty}

\vspace*{1cm}
\begin{center}
{\Large\bf Jordan chains of elliptic partial differential operators \\[2mm]
and Dirichlet-to-Neumann maps} \\[5mm]

\large J. Behrndt$^1$ and A.F.M. ter Elst$^2$

\end{center}

\vspace{5mm}

\begin{center}
{\bf Abstract}
\end{center}

\begin{list}{}{\leftmargin=1.8cm \rightmargin=1.8cm \listparindent=10mm 
   \parsep=0pt}
\item
Let $\Omega \subset \Ri^d$ be a bounded open set with Lipschitz boundary $\Gamma$. 
It will be shown that the Jordan chains of m-sectorial second-order 
elliptic partial differential operators
with measurable coefficients and (local or non-local) Robin boundary conditions in 
$L_2(\Omega)$ can be characterized with the help of Jordan chains
of the Dirichlet-to-Neumann map and the boundary operator from 
$H^{1/2}(\Gamma)$ into $H^{-1/2}(\Gamma)$. 
This result extends the Birman--Schwinger principle in the framework
of elliptic operators 
for the characterization of eigenvalues, eigenfunctions and geometric eigenspaces 
to the complete set of all generalized eigenfunctions and algebraic
eigenspaces.

\end{list}

\vspace{2cm}
\noindent
May 2019

\vspace{2mm}

\vspace{2mm}

\vspace{5mm}

\noindent
{\bf Home institutions:}    \\[3mm]
\begin{tabular}{@{}cl@{\hspace{10mm}}cl}
1. & Institut f\"ur Angewandte Mathematik  & 
  2. & Department of Mathematics   \\
& Technische Universit\"at Graz   & 
  & University of Auckland   \\
& Steyrergasse 30 & 
  & Private bag 92019  \\
& A-8010 Graz & 
  & Auckland 1142 \\
& Austria  & 
  & New Zealand  \\[8mm]
& Email: behrndt@tugraz.at  &
  & Email: terelst@math.auckland.ac.nz 
\end{tabular}

\newpage

\bibliographystyle{tom}

\section{Introduction} \label{Sjordanchain1}

The Dirichlet-to-Neumann map is an important object in the analysis of elliptic partial 
differential equations 
since it can be used to describe the spectra of the associated elliptic operators. 
The principal strategy and advantage is that a spectral problem for a 
partial differential operator on a domain 
$\Omega$ is reduced to a spectral problem for an operator function on the boundary $\Gamma$ 
of this domain, where, very roughly speaking,
the Dirichlet and Neumann data can be {\it measured}.
This type of approach to problems in spectral and scattering theory for 
elliptic partial differential operators was used in the self-adjoint case in, e.g.\ 
\cite{ArM2,BMN,BehR,BehR3,GesM2,GesM3,GesMZ2,MPavlovP,Marl,MikhailovaPPRY,MPavlovP,Post}, for 
non-self-adjoint situations in, e.g.\ \cite{BGHN,BGW,Grubb1,Mal},
and we also refer the reader to the 
more abstract contributions 
\cite{AE3, AE6, AE7, AEKS, AEW,BrMN,BrownHMNW,BrownMarlettaNW,BruningGeykerP,
DemuthHK,DerkachM,DerkachM2,EO6, EO7,LanT,MalamudMog,Posilicano}.

In the present paper we are interested in a characterization of Jordan chains of eigenvalues of elliptic operators. 
To motivate our investigations let us consider here in the introduction only the
special case
of a Schr\"{o}dinger operator $\cA=-\Delta+V$ on a bounded Lipschitz domain 
$\Omega\subset\Ri^d$ with $d\geq 2$ and with a complex-valued
potential $V \in L_\infty(\Omega)$.
Later in this paper much more general second-order partial differential expressions $\cA$ with 
measurable coefficients will be considered; see Section~\ref{Sjordanchain2} for details.
The Dirichlet-to-Neumann map $D(\lambda)$ corresponding to $-\Delta+V$ can be defined as 
a bounded operator $D(\lambda)\colon H^{1/2}(\Gamma)\rightarrow H^{-1/2}(\Gamma)$ by
\begin{equation*}
 \Tr f_\lambda \mapsto \gamma_N f_\lambda,
\end{equation*}
where $f_\lambda\in H^1(\Omega)$ is such that $\cA f_\lambda=\lambda f_\lambda$.
Here $\Tr f_\lambda\in H^{1/2}(\Gamma)$ and $\gamma_N f_\lambda\in H^{-1/2}(\Gamma)$
denote the Dirichlet and Neumann trace of $f_\lambda$, respectively, 
and $\lambda\in\dC$ is not an eigenvalue of the Dirichlet
realization $A_D$ of $-\Delta + V$.
Assume for simplicity that $B\colon L_2(\Gamma)\rightarrow L_2(\Gamma)$ is a bounded 
operator and consider the (non-local) Robin realization of $-\Delta+V$ defined 
by
\begin{equation}\label{abintro}
A_B f= -\Delta f+Vf,\quad 
\dom A_B=\bigl\{f\in H^1(\Omega): 
        \gamma_N f=B\Tr f \mbox{ and } -\Delta f+Vf\in L_2(\Omega)\bigr\}. 
\end{equation}
Note that the resolvents of $A_D$ and $A_B$ are both compact operators in $L_2(\Omega)$ 
due to the compactness of the embedding $H^1(\Omega)\hookrightarrow L_2(\Omega)$
and hence the spectra of $A_D$ and $A_B$ are discrete.
It is well-known and easy to see that for all $\lambda_0\not\in\sigma_p(A_D)$ 
one has $\lambda_0\in\sigma_p(A_B)$ if and only if $\ker (D(\lambda_0)-B)\not= \{0\}$.
Sometimes this is referred to as a variant of the Birman--Schwinger principle.
In fact, if $\lambda_0\in\sigma_p(A_B)$ and $f_0\in\dom A_B$ is a 
corresponding eigenfunction, then $\Tr f_0\not=0 $ (as otherwise $f_0$
would be an eigenfunction for $A_D$ at $\lambda_0$) and 
\begin{equation*}
 (D(\lambda_0)-B)\Tr f_0=D(\lambda_0)\Tr f_0 - B \Tr f_0=\gamma_N f_0- B \Tr f_0=0,
\end{equation*}
and conversely, if $\varphi\in\ker(D(\lambda_0)-B) \setminus \{ 0 \} $,
then the unique solution $f_0\in H^1(\Omega)$ of the boundary value problem 
$(-\Delta +V)f_0=\lambda_0 f_0$ with $\Tr f_0=\varphi$, satisfies $\gamma_N f_0- B \Tr f_0=0$, so that
$f_0\in\dom A_B$ is an eigenfunction of $A_B$ corresponding to $\lambda_0$. 

In the situation where the potential $V$ is not real-valued or the Robin boundary 
operator~$B$ is not symmetric the Schr\"{o}dinger operator $A_B$ in \eqref{abintro}
is m-sectorial, but not self-adjoint in $L_2(\Omega)$.
Therefore, in general, the eigenvalues of $A_B$ are not semisimple and 
besides an eigenvector $f_0$ 
also (finitely many) generalized eigenvectors $f_1,\ldots, f_k$ are associated to an 
eigenvalue $\lambda_0$, which form a so-called Jordan chain.
It is the main objective of the present paper to analyse the
Jordan chains $f_0,f_1,\ldots,f_k$ corresponding to an eigenvalue $\lambda_0$ of $A_B$ 
with the help of the Dirichlet-to-Neumann operator in a similar form
as in the above mentioned Birman--Schwinger principle.
In fact, using the notion of Jordan chains for holomorphic operator functions due to 
M.V. Keldysh \cite{Keldysh} 
(see also \cite[$\S$11]{Markus}), it turns out in our main result
Theorem~\ref{main}
that $ \{ f_0,f_1,\ldots,f_k \} $ form a Jordan chain of $A_B$ at 
$\lambda_0\in\sigma_p(A_B)\cap\rho(A_D)$ if and only if the corresponding 
traces $\varphi_0=\Tr f_0,\varphi_1=\Tr f_1,\ldots,\varphi_k=\Tr f_k$ form a Jordan chain 
for the holomorphic  
$\cL(H^{1/2}(\Gamma),H^{-1/2}(\Gamma))$-valued operator function
$\lambda\mapsto M(\lambda)=D(\lambda)-B$ at $\lambda_0$, that is, 
\begin{equation}\label{intromain}
 \sum_{l=0}^j\frac{1}{l!} M^{(l)}(\lambda_0)\varphi_{j-l}=0
\end{equation}
for all $j \in \{ 0,\ldots,k \} $,
where $M^{(l)}(\lambda_0)$ denotes the $l$-th derivative of the function $M$ at $\lambda_0$.
Note that for $j=0$ the characterization of the eigenvector $f_0$ in the 
Birman--Schwinger principle follows from \eqref{intromain}; see the above discussion
or Corollary~\ref{maincor}.

The structure of this paper is as follows.
In Section~\ref{sec2} we briefly recall the notion of Jordan chains for operators and 
holomorphic operator functions.
In Section~\ref{Sjordanchain2} we introduce the elliptic differential operators 
and the corresponding
Dirichlet-to-Neumann map that is used for the analysis of the algebraic eigenspaces. 
Here we treat second-order divergence form elliptic operators with (complex) 
$L_\infty$-coefficients of the form 
\begin{equation*}
 \cA=- \sum_{k,l=1}^d \partial_k c_{kl} \partial_l  
   + \sum_{k=1}^d c_k \partial_k 
   -\sum_{k=1}^d \partial_k b_k  +c_0
\end{equation*}
on
bounded Lipschitz domains with non-local Robin boundary conditions.
In this general situation it is necessary to pay 
special attention to the definition and properties of the co-normal and adjoint 
co-normal derivative, and to the properties of the corresponding sesquilinear forms and operators.
Furthermore, 
the unique solvability of the homogeneous and inhomogeneous Dirichlet boundary value problems 
is discussed. For the convenience of the reader we provide proofs of these preparatory results in Section~\ref{Sjordanchain2}.
Our main result on the
characterization of Jordan chains of second-order elliptic partial differential operators 
with local or non-local Robin boundary conditions via Jordan chains of the 
Dirichlet-to-Neumann map $\lambda\mapsto D(\lambda)$ and the boundary operator $B$ is
formulated and proved in Section~\ref{sec4}.
The proof is technical and requires the preparatory Lemma~\ref{mainlem}.
Finally, in Subsection~\ref{Subjordanchain5.1} we
discuss a more regular situation in which the bounded domain $\Omega$ is assumed to have 
a $C^2$-smooth boundary and the coefficients of the
elliptic operator are slightly more regular.
In this setting one then obtains a Dirichlet-to-Neumann operator acting from 
$H^{3/2}(\Gamma)$ into $H^{1/2}(\Gamma)$ and a variant of Theorem~\ref{main} 
for $H^2(\Omega)$-smooth Jordan chains.
In Subsection~\ref{Subjordanchain5.2} we reconsider the Dirichlet-to-Neumann operator
on a Lipschitz domain, but now we treat the 
Dirichlet-to-Neumann operator acting from $H^1(\Gamma)$ into $L_2(\Gamma)$.
For this we require a smoothness and symmetry condition on the principal 
coefficients.

\paragraph*{Acknowledgements.} 
J. Behrndt is most grateful for the stimulating research
stay and the hospitality at the University of Auckland, where
parts of this paper were written. This work is supported by the Austrian Science
Fund (FWF), project P 25162-N26 and part of this work is supported by the
Marsden Fund Council from Government funding, administered by the Royal
Society of New Zealand.

\section{Jordan chains of operators and holomorphic operator functions}\label{sec2}

Throughout this paper the field is the complex numbers.
Let $A$ be a linear operator in a Banach space $\cH$.
Further, let $k\in\dN_0$, $f_0,\ldots,f_k \in \cH$ and $\lambda_0\in\dC$.
Then we say that the vectors $\{f_0,\ldots,f_k\}$ form a {\bf Jordan chain} for $A$ at 
$\lambda_0$ if $f_j\in\dom A$ for all $j \in \{ 0,\ldots,k \} $ satisfy 
\begin{equation*}
(A-\lambda_0) f_j=f_{j-1}
\end{equation*}
for all $j \in \{ 0,\ldots,k \} $
with $f_0\not=0$ and we set $f_{-1} = 0$.
The vector $f_0$ is called an {\bf eigenvector} of $A$ at the {\bf eigenvalue} 
$\lambda_0$ and the vectors $f_1,\ldots, f_k$ are said to be {\bf generalized eigenvectors}
of $A$ at $\lambda_0$.
Note that the generalized eigenvectors are all nonzero.

The notion of Jordan chains exists also for holomorphic operator functions and goes 
back to the work of M.V. Keldysh \cite{Keldysh}, for more details we also refer the reader
to the monograph \cite[$\S$11]{Markus}.
Let $\cH_1$ and $\cH_2$ be Banach spaces, $\cO\subset\dC$ an open set 
and for all $\lambda\in\cO$ let $M(\lambda) \in \cL(\cH_1,\cH_2)$.
Assume, in addition, that the operator function $\lambda\mapsto M(\lambda)$ is 
holomorphic on $\cO$ and denote the $l$-th derivative of $M(\cdot)$ at $\lambda\in\cO$
by $M^{(l)}(\lambda)$.
Let $k \in \dN_0$ and $\varphi_0,\ldots,\varphi_k \in \cH_1$.
Then we say that the vectors $\{\varphi_0,\ldots,\varphi_k\}$  
form a {\bf Jordan chain} for the function $M(\cdot)$ at $\lambda_0\in\cO$ if 
\begin{equation*}
 \sum_{l=0}^j\frac{1}{l!} M^{(l)}(\lambda_0)\varphi_{j-l}=0
\end{equation*}
for all $j \in \{ 0,\ldots,k \} $ and $\varphi_0\not=0$.
The vector $\varphi_0$ is called an {\bf eigenvector} of the operator function 
$M(\cdot)$ at the {\bf eigenvalue} $\lambda_0$ 
and the vectors $\varphi_1,\ldots, \varphi_k$ are said to be {\bf generalized eigenvectors}
of $M(\cdot)$ at $\lambda_0$. 

Observe that in the special case $\cH_1=\cH_2$ and $C\in\cL(\cH_1)$ the notion of Jordan chain for the operator $C$ at $\lambda_0\in\dC$
and the notion of Jordan chain for the function $\lambda\mapsto C-\lambda$ at $\lambda_0\in\dC$ coincide.

\section{Elliptic differential operators and Dirichlet-to-Neu\-mann maps} \label{Sjordanchain2}

Let $\Omega \subset \Ri^d$ be a bounded Lipschitz domain with boundary $\Gamma$.
By $H^1(\Omega)$ we denote the $L_2$-based Sobolev space of order $1$ on $\Omega$
and $H_0^1(\Omega)$ denotes the closure of the compactly supported $C_c^\infty(\Omega)$-functions 
in $H^1(\Omega)$.
On the Lipschitz boundary $\Gamma$ the Sobolev space  
$H^{1/2}(\Gamma)$ of order $1/2$ will play an important role.
Its dual is denoted by $H^{-1/2} (\Gamma)$ and $\langle\cdot,\cdot\rangle$ stands for 
the extension of the $L_2(\Gamma)$ inner product
onto the pair $H^{1/2} (\Gamma)\times H^{-1/2} (\Gamma)$.
Recall from \cite{McL} Theorem~3.37 that there is a continuous {\bf trace map}
$\Tr \colon H^1(\Omega) \to H^{1/2}(\Gamma)$ 
such that $\Tr f = f|_\Gamma$ for all $f \in H^1(\Omega) \cap C^1(\overline \Omega)$
and it admits a bounded right inverse.

For all $k,l \in \{ 1,\ldots,d \} $ fix $c_{kl}, b_k, c_k, c_0 \in L_\infty(\Omega)$.
We recall that the field is the complex numbers, so we emphasise that all 
coefficients are complex valued.
Assume that there exists a $\mu > 0$ such that 
\[
\RRe \sum_{k,l=1}^d c_{kl}(x) \, \xi_k \, \overline{\xi_l}
\geq \mu \, |\xi|^2
\]
for all $x \in \Omega$ and $\xi \in \Ci^d$.
Define the sesquilinear form $\gota \colon H^1(\Omega) \times H^1(\Omega) \to \Ci$
by 
\[
\gota(f,g)
= \sum_{k,l=1}^d \int_\Omega c_{kl} (\partial_l f) \overline{\partial_k g}
   + \sum_{k=1}^d \int_\Omega c_k (\partial_k f) \overline g
   + \sum_{k=1}^d \int_\Omega b_k f \overline{\partial_k g}
   + \int_\Omega c_0 f \overline g 
.  \]
The form $\gota$ is continuous in the sense that 
there exists an $M \geq 0$ such that 
$|\gota(f,g)| \leq M \, \|f\|_{H^1(\Omega)} \, \|g\|_{H^1(\Omega)}$
for all $f,g \in H^1(\Omega)$. One verifies in the same way as in the proof of \cite{AE1} Lemma~3.7 that the form is elliptic and hence 
\cite{AE2} Lemma~3.1 implies that $\gota$ is a closed sectorial form.

Introduce $\cA \colon H^1(\Omega) \to (H^1_0(\Omega))^*$ by 
\[
\langle\cA f,g\rangle_{(H^1_0(\Omega))^* \times H^1_0(\Omega)}
= \gota(f,g)
.
\]

In order to introduce the co-normal derivative we need a lemma.
Note that the ellipticity condition on the principal coefficients is not 
needed in the next lemma.

\begin{lemma} \label{ljordanchain310}
Let $f \in H^1(\Omega)$ and suppose that $\cA f \in L_2(\Omega)$.
Then there exists a unique $\psi \in H^{-1/2}(\Gamma)$ such that 
\[
\gota(f,g) - (\cA f, g)_{L_2(\Omega)}
= \langle \psi, \Tr g \rangle_{H^{-1/2}(\Gamma) \times H^{1/2}(\Gamma)}
\]
for all $g \in H^1(\Omega)$.
Moreover, there exists a constant $c > 0$, independent of $f$, such that 
$\|\psi\|_{H^{-1/2}(\Gamma)} \leq c (\|f\|_{H^1(\Omega)} + \|\cA f\|_{L_2(\Omega)})$.
\end{lemma}

\begin{proof}
Define $F \colon H^1(\Omega) \to \Ci$ by 
$F(g) = \gota(f,g) - (\cA f, g)_{L_2(\Omega)}$.
Then $F$ is anti-linear and bounded. 
Explicitly, there exists an $M \geq 0$, independent of $f$, such that 
$$\|F\|_{H^1(\Omega)^*} \leq M \, \|f\|_{H^1(\Omega)} + \|\cA f\|_{L_2(\Omega)}.$$
Moreover, $F(g) = 0$ for all $g \in H^1_0(\Omega)$.
Hence there exists a unique anti-linear $\widetilde F \colon H^{1/2}(\Gamma) \to \Ci$
such that $\widetilde F(\Tr g) = F(g)$ for all $g \in H^1(\Omega)$.
The map $\widetilde F$ is bounded and 
$\|\widetilde F\|_{H^{/2}(\Gamma)^*} \leq \|F\|_{H^1(\Omega)^*} \, \|Z\|$,
where $Z \colon H^{1/2}(\Gamma) \to H^1(\Omega)$ is a bounded right 
inverse of $\Tr$.
Write $\psi = \widetilde F \in H^{1/2}(\Gamma)^* = H^{-1/2}(\Gamma)$.
Then
$\widetilde F(\varphi) = \langle \psi, \varphi\rangle_{H^{-1/2}(\Gamma) \times H^{1/2}(\Gamma)}$
for all $\varphi \in H^{1/2}(\Gamma)$ and the lemma follows. 
\end{proof}

If $f \in H^1(\Omega)$ with $\cA f \in L_2(\Omega)$, then we denote 
by $\gamma_N f \in H^{-1/2}(\Gamma)$ the function such that
\[
\gota(f,g) - (\cA f, g)_{L_2(\Omega)}
= \langle \gamma_N f, \Tr g \rangle_{H^{-1/2}(\Gamma) \times H^{1/2}(\Gamma)}
\]
for all $g \in H^1(\Omega)$.
We call $\gamma_N f$ the {\bf co-normal derivative} of $f$.

Denote by $\gota_D$ the restriction of $\gota$ to $H^1_0(\Omega) \times H^1_0(\Omega)$.
Then $\gota_D$ is a continous elliptic form and hence a closed sectorial form
(cf.\ \cite{AE2} Lemma~3.1.)
Denote by  $A_D$ the m-sectorial operator associated with the form $\gota_D$.
It follows that $A_D$ is the Dirichlet realization of $\cA$ in $L_2(\Omega)$
given by
\begin{equation*}
A_D f = \cA f,\quad \dom A_D = \bigl\{ f \in H^1_0(\Omega) : \cA f \in L_2(\Omega)\bigr\}.
\end{equation*}

\begin{lemma} \label{sollem}
Let $\lambda\in\rho(A_D)$.
Then the following assertions hold.
\begin{tabel}
\item \label{sollem-1}
For all $\varphi\in H^{1/2}(\Gamma)$
there exists a unique solution $f\in H^1(\Omega)$ of the homogeneous boundary value problem
\begin{equation}\label{hom}
 (\cA-\lambda)f=0
\quad \text{and} \quad \Tr f=\varphi.
\end{equation}
Moreover, the map $\varphi \mapsto f$ is continuous from $H^{1/2}(\Gamma)$ into 
$H^1(\Omega)$.
\item \label{sollem-2} 
For all $\varphi\in H^{1/2}(\Gamma)$ and all $h\in L_2(\Omega)$ 
there exists a unique solution $f\in H^1(\Omega)$ of the inhomogeneous boundary value problem
\begin{equation}\label{inhom}
 (\cA-\lambda)f=h
\quad \text{and} \quad 
\Tr f=\varphi.
\end{equation}
\end{tabel}
\end{lemma}

\begin{proof}
`\ref{sollem-1}'.
The existence follows as in the proof of \cite{AE9} Lemma~2.1.
For completeness we give the details.
There exists a $T \in \cL(H^1_0(\Omega))$ such that 
\[
(T f,g)_{H^1_0(\Omega)} = \gota_D(f,g) - \lambda (f,g)_{L_2(\Omega)}
\]
for all $f,g \in H^1_0(\Omega)$.
Further there exists an $\omega > 0$ such that 
the sesquilinear form $\gotb \colon H^1_0(\Omega) \times H^1_0(\Omega) \to \Ci$
given by $\gotb(f,g) = \gota_D(f,g) - \lambda (f,g)_{L_2(\Omega)} + \omega (f,g)_{L_2(\Omega)}$
is coercive.
Let $j \colon H^1_0(\Omega) \to L_2(\Omega)$ be the (compact) inclusion map.
Then $\gotb(f,g) = ((T + K) f,g)_{H^1_0(\Omega)}$ 
for all $f,g \in H^1_0(\Omega)$, where $K = \omega j^* j$.
So $T+K$ is invertible by the Lax--Milgram theorem.
Consequently $T$ is a Fredholm operator because $K$ is compact.
Now $T$ is injective since $\lambda \in \rho(A_D)$.
Hence $T$ is surjective.

There exists an $f_0 \in H^1(\Omega)$ such that $\Tr f_0 = \varphi$.
Hence there exists an $h \in H^1_0(\Omega)$ such that 
$(T h,g)_{H^1_0(\Omega)} = \gota(f_0,g) - \lambda (f_0,g)_{L_2(\Omega)}$ 
for all $g \in H^1_0(\Omega)$.
Then $f = f_0 - h$ satisfies
\[
\langle\cA f-\lambda f,g\rangle_{(H^1_0(\Omega))^* \times H^1_0(\Omega)}
= \gota(f_0,g)-\lambda(f_0,g)_{L_2(\Omega)} - \gota_D(h,g)+\lambda(h,g)_{L_2(\Omega)}=0
\]
and hence $(\cA-\lambda)f=0$. 
The uniqueness is easy.
The continuity of the map follows from the closed graph theorem.

`\ref{sollem-2}'.
By Statement~\ref{sollem-1} there exists an $f_0 \in H^1(\Omega)$ such that 
$(\cA - \lambda) f_0 = 0$ and $\Tr f_0 = \varphi$.
Then $f_0 + (A_D - \lambda)^{-1} h$ is a solution to the problem~(\ref{inhom}).
Again the uniqueness is easy.
\end{proof}

Let $\lambda \in \rho(A_D)$.
Now we are able to define the {\bf Dirichlet-to-Neumann operator} 
$D(\lambda) \colon H^{1/2}(\Gamma) \to H^{-1/2}(\Gamma)$.
Let $\varphi \in H^{1/2}(\Gamma)$.
By Lemma~\ref{sollem}\ref{sollem-1} there exists a unique solution
$f \in H^1(\Omega)$ of the homogeneous boundary value problem
\eqref{hom}.
Then $\cA f = \lambda f \in L_2(\Omega)$.
Hence one can define 
\[
D(\lambda) \varphi = \gamma_N f
.  \]
Then $D(\lambda)$ is bounded operator from $H^{1/2}(\Gamma)$ into $H^{-1/2}(\Gamma)$
by the last parts of Lemmas~\ref{ljordanchain310} and \ref{sollem}\ref{sollem-1}.

We need two holomorphy results.

\begin{lemma} \label{ljordanchain201}
\mbox{}
\begin{tabel}
\item \label{ljordanchain201-1}
Let $\varphi \in H^{1/2}(\Gamma)$.
For all $\lambda \in \rho(A_D)$ let $g_\lambda \in H^1(\Omega)$ 
be the unique element such that 
$(\cA - \lambda) g_\lambda = 0$ and $\Tr g_\lambda = \varphi$.
Then the map $\lambda \mapsto g_\lambda$ is holomorphic from 
$\rho(A_D)$ into $H^1(\Omega)$.
\item \label{ljordanchain201-2}
The map $\lambda \mapsto D(\lambda)$ is holomorphic from 
$\rho(A_D)$ into $\cL(H^{1/2}(\Gamma),H^{-1/2}(\Gamma))$.
\end{tabel}
\end{lemma}

\begin{proof}
`\ref{ljordanchain201-1}'.
Fix $\lambda_0\in\rho(A_D)$.
By Lemma~\ref{sollem}\ref{sollem-1} there exists a 
unique $g_{\lambda_0}\in H^1(\Omega)$ such that 
$(\cA - \lambda_0) g_{\lambda_0} = 0$ and $\Tr g_{\lambda_0} = \varphi$.
Let $\lambda\in\rho(A_D)$ and consider
\begin{equation}\label{gres}
 g=\bigl(1+(\lambda-\lambda_0)(A_D-\lambda)^{-1}\bigr)g_{\lambda_0}\in H^1(\Omega).
\end{equation}
Then $(\cA-\lambda)g=(\cA-\lambda)g_{\lambda_0}+(\lambda-\lambda_0)g_{\lambda_0}=0$ 
and $\Tr g=\Tr g_{\lambda_0}=\varphi$.
Since the solution 
of the homogeneous boundary value problem $(\cA-\lambda)f=0$ with $\Tr f=\varphi$, 
is unique by Lemma~\ref{sollem}\ref{sollem-1} it follows that $g=g_\lambda$.
Now the holomorphy of the resolvent $\lambda\mapsto (A_D-\lambda)^{-1}$ in \eqref{gres} 
implies that the map $\lambda \mapsto g_\lambda$ is holomorphic from 
$\rho(A_D)$ into $H^1(\Omega)$.

`\ref{ljordanchain201-2}'.
Let $\varphi \in H^{1/2}(\Gamma)$ and $h \in H^1(\Omega)$.
For all $\lambda \in \rho(A_D)$ let 
$g_\lambda \in H^1(\Omega)$ be as in Statement~\ref{ljordanchain201-1}.
Then 
\begin{eqnarray*}
\langle D(\lambda) \varphi,\Tr h \rangle_{H^{-1/2}(\Gamma) \times H^{1/2}(\Gamma)}
& = & \langle \gamma_N g_\lambda, \Tr h \rangle_{H^{-1/2}(\Gamma) \times H^{1/2}(\Gamma)}  \\
& = & \gota(g_\lambda, h) - (\cA g_\lambda, h)_{L_2(\Omega)}  \\
& = & \gota(g_\lambda, h) - \lambda (g_\lambda, h)_{L_2(\Omega)}
\end{eqnarray*}
for all $\lambda \in \rho(A_D)$.
Since $\lambda \mapsto g_\lambda$ is holomorphic from 
$\rho(A_D)$ into $H^1(\Omega)$ by Statement~\ref{ljordanchain201-1}, it follows that 
$\lambda \mapsto D(\lambda)$ is holomorphic with respect to the 
weak operator topology on $\cL(H^{1/2}(\Gamma),H^{-1/2}(\Gamma))$, and therefore
it is also holomorphic with respect to the uniform operator topology.
\end{proof}

For all $l \in \Ni$ we denote the $l$-th derivative 
of $\lambda \mapsto D(\lambda)$ at $\lambda\in\rho(A_D)$ by $D^{(l)}(\lambda)$.
Then according to Lemma~\ref{ljordanchain201}\ref{ljordanchain201-2} one has 
\[
D^{(l)}(\lambda) \in \cL(H^{1/2}(\Gamma),H^{-1/2}(\Gamma))
\]
for all $\lambda\in\rho(A_D)$.

The {\bf dual form} $\gota^*$ of $\gota$ is defined by 
$\dom(\gota^*) = H^1(\Omega)$ and $\gota^*(f,g) = \overline{\gota(g,f)}$
for all $f,g \in H^1(\Omega)$.
So 
\[
\gota^*(f,g)
= \sum_{k,l=1}^d \int_\Omega \overline{c_{lk}} (\partial_l f) \overline{\partial_k g}
   + \sum_{k=1}^d \int_\Omega \overline{b_k} (\partial_k f) \overline g
   + \sum_{k=1}^d \int_\Omega \overline{c_k} f \overline{\partial_k g}
   + \int_\Omega \overline{c_0} f \overline g 
.  \]
Obviously $\gota^*$ is of the same type as $\gota$, with $c_{kl}$ replaced 
by $\overline{c_{lk}}$, etc.
Similar to the definition of $\cA$ with respect to $\gota$, we can define
the operator $\widetilde \cA \colon H^1(\Omega) \to (H^1_0(\Omega))^*$ by 
\[
\langle \widetilde \cA f,g\rangle_{(H^1_0(\Omega))^* \times H^1_0(\Omega)}
= \gota^*(f,g)
.
\]
As in Lemma~\ref{ljordanchain310} it follows that 
for all $f \in H^1(\Omega)$ with $\widetilde \cA f \in L_2(\Omega)$, there 
exists a unique $\widetilde \gamma_N f \in H^{-1/2}(\Gamma)$ such that
\[
\gota^*(f,g) - (\widetilde \cA f, g)_{L_2(\Omega)}
= \langle \widetilde \gamma_N f, \Tr g \rangle_{H^{-1/2}(\Gamma) \times H^{1/2}(\Gamma)}
\]
for all $g \in H^1(\Omega)$.
Using all definitions it is easy to prove the following version of 
Green's second identity.

\begin{lemma} \label{ljordanchain311}
Let $f,g \in H^1(\Omega)$ and suppose that $\cA f,\widetilde\cA g\in L_2(\Omega)$.
Then
\begin{equation}\label{green2}
(\cA f,g)_{L_2(\Omega)} - (f,\widetilde\cA g)_{L_2(\Omega)}
= \langle \Tr f,\widetilde\gamma_N g\rangle_{H^{1/2}(\Gamma) \times H^{-1/2}(\Gamma)}
   - \langle\gamma_N f,\Tr g\rangle_{H^{-1/2}(\Gamma) \times H^{1/2}(\Gamma)}
.
\end{equation}
\end{lemma}

Denote by $\gota_D^*$ the restriction of the dual form $\gota^*$ to $H^1_0(\Omega) \times H^1_0(\Omega)$.
Then $\gota_D^*$ is a closed sectorial form and 
the m-sectorial operator associated with $\gota_D^*$ is equal to the 
adjoint $A_D^*$ of $A_D$, see \cite{Kat1} Theorem~VI.2.5. 
It follows that $A_D^*$ is the Dirichlet realization of $\widetilde\cA$ in $L_2(\Omega)$
given by
\begin{equation*}
A_D^* f = \widetilde \cA f,\quad \dom A_D^* = \bigl\{ f \in H^1_0(\Omega) : \widetilde\cA f \in L_2(\Omega)\bigr\}.
\end{equation*}

Similarly to the Dirichlet-to-Neumann map $D(\lambda)\in\cL(H^{1/2}(\Gamma),H^{-1/2}(\Gamma))$ 
one associates the Dirichlet-to-Neumann map
$\widetilde D(\lambda)\in\cL(H^{1/2}(\Gamma),H^{-1/2}(\Gamma))$ to the adjoint form
$\gota^*$ for all $\lambda \in \rho(A_D^*)$. 
A simple computation based on Greens second identity \eqref{green2} shows 
\begin{equation}
\langle D(\lambda)\varphi,\psi\rangle_{H^{-1/2}(\Gamma) \times H^{1/2}(\Gamma)}
 =\langle\varphi,\widetilde D(\overline \lambda)\psi\rangle_{H^{1/2}(\Gamma) \times H^{-1/2}(\Gamma)}
\label{eGreenSecond}
\end{equation}
for all $\varphi,\psi\in H^{1/2}(\Gamma)$ and $\lambda\in\rho(A_D)$.

Finally we introduce the Robin operator.
Let $B\in\cL(H^{1/2}(\Gamma),H^{-1/2}(\Gamma))$.
We assume that there is an $\eta > 0$ such that 
\begin{equation}\label{bsemi}
\RRe\langle B \varphi,\varphi\rangle_{H^{-1/2}(\Gamma)\times   H^{1/2}(\Gamma)}
\leq \eta \|\varphi\|_{L_2(\Gamma)}^2
\end{equation}
for all $\varphi\in H^{1/2}(\Gamma)$.
Note that the restriction to the space $H^{1/2}(\Gamma)$ 
of every bounded operator $B$ in $L^2(\Gamma)$ 
can be viewed as an operator in 
$\cL(H^{1/2}(\Gamma),H^{-1/2}(\Gamma))$ that satisfies \eqref{bsemi}.
We also note that
the above assumption on $B\in\cL(H^{1/2}(\Gamma),H^{-1/2}(\Gamma))$ can be generalized 
further as in for example \cite{GesM} Hypothesis~4.1.
Next we define the 
sesquilinear form $\gota_B \colon H^1(\Omega) \times H^1(\Omega) \to \Ci$
by 
\[
\gota_B(f,g) = \gota(f,g) - \langle B\Tr f,\Tr g\rangle_{H^{-1/2}(\Gamma)\times H^{1/2}(\Gamma)}
.  
\]

\begin{proposition}\label{ljordan2-313}
The form $\gota_B$ is densely defined,  closed and sectorial in $L_2(\Omega)$. 
The associated m-sectorial operator 
\[
 A_B f = \cA f,\qquad 
\dom A_B=\bigl\{f\in H^1(\Omega): \cA f\in L_2(\Omega) \mbox{ and } \gamma_N f =B\Tr f \bigr\},
\]
is the {\bf Robin realisation} of $\cA$ in $L_2(\Omega)$. 
\end{proposition}

\begin{proof}
We will show first that $\gota_B$ is elliptic, that is,
there are $\nu \in \Ri$ and $\mu > 0$ such that 
\begin{equation} \label{elliab}
 \RRe \gota_B(f) + \nu \|f\|_{L_2(\Omega)}^2 \geq \mu \|f\|_{H^1(\Omega)}^2
\end{equation}
for all $f\in H^1(\Omega)$.
Clearly there are $\mu_1,\omega_1 > 0$ such that 
$\RRe \gota(f) \geq 2 \mu_1 \|f\|_{H^1(\Omega)}^2 - \omega_1 \|f\|_{L_2(\Omega)}^2$
for all $f \in H^1(\Omega)$ (cf.\ \cite{AE1} Lemma~3.7.)
Choose $\varepsilon <\frac{\mu_1}{\eta}$, where $\eta > 0$ is as in \eqref{bsemi}.
By Ehrlings lemma and the compactness of $\Tr \colon H^1(\Omega) \to L_2(\Gamma)$
there exists a $c > 0$ such that 
$\|\Tr f\|_{L_2(\Gamma)}^2 \leq \varepsilon \|f\|_{H^1(\Omega)}^2 + c \|f\|_{L_2(\Omega)}^2$
for all $f \in H^1(\Omega)$.
Then 
\[
\RRe\langle B \Tr f,\Tr f\rangle_{H^{-1/2}(\Gamma)\times H^{1/2}(\Gamma)}
\leq \eta \|\Tr f\|_{L_2(\Gamma)}^2
\leq \mu_1 \|f\|_{H^1(\Omega)}^2 + \eta c \|f\|_{L_2(\Omega)}^2
\]
and hence
\[
\RRe \gota_B(f) =\RRe \gota(f)-\RRe\langle B \Tr f,\Tr f\rangle_{H^{-1/2}(\Gamma)\times H^{1/2}(\Gamma)}
\geq \mu_1 \|f\|_{H^1(\Omega)}^2 
     - (\omega_1+\eta c) \|f\|_{L_2(\Omega)}^2
\]
for all $f \in H^1(\Omega)$. 
So \eqref{elliab} holds with $\mu=\mu_1$ and 
$\nu=\omega_1+\eta c$, therefore $\gota_B$ is elliptic.
Hence $\gota_B$ is a densely defined, closed, sectorial form (see \cite{AE2} Lemma~3.1). 

The graph of the m-sectorial operator associated to $\gota_B$ is given by
\[
G=\bigl\{ (f,h) \in H^1(\Omega) \times L_2(\Omega) : 
      \gota_B(f,g) = (h,g)_{L_2(\Omega)} \mbox{ for all } g \in H^1(\Omega) \bigr\}
 \]
 and it remains to show that $G$ coincides with the Robin realisation $A_B$. 
Now let $f \in \dom G$ and 
write $h = G f \in L_2(\Omega)$.
Then $f \in H^1(\Omega)$ and 
\[
\langle\cA f, g\rangle_{(H^1_0(\Omega))^* \times H^1_0(\Omega)}
= \gota(f,g) 
= \gota_B(f,g)
= (h,g)_{L_2(\Omega)}
\]
for all $g \in H^1_0(\Omega)$.
So $\cA f = h = Gf \in L_2(\Omega)$.
If $g \in H^1(\Omega)$, then
\begin{eqnarray*}
\gota(f,g) - (\cA f,g)_{L_2(\Omega)}
& = & \gota_B(f,g) + \langle B \Tr f,\Tr g\rangle_{H^{-1/2}(\Gamma)\times H^{1/2}(\Gamma)} - (h,g)_{L_2(\Omega)}  \\
& = & \langle B \Tr f,\Tr g\rangle_{H^{-1/2}(\Gamma)\times H^{1/2}(\Gamma)}.
\end{eqnarray*}
So $\gamma_N f = B \Tr f$ and hence $f\in\dom A_B$.
The converse inclusion follows similarly.
\end{proof}

\section{Jordan chains of Robin realizations}\label{sec4}

Adopt the assumptions and notation as in Section~\ref{Sjordanchain2}.
In this section we formulate and prove our main result on the characterization of Jordan chains of the m-sectorial 
Robin realization $A_B$ of $\cA$ via the operator function $\lambda\mapsto D(\lambda)-B$.
Our goal is to show the following theorem.

\begin{theorem}\label{main}
Let $A_B$ be the Robin realisation of $\cA$ in $L_2(\Omega)$ as 
in Proposition~\ref{ljordan2-313}, let $\lambda_0\in\rho(A_D)$ and consider the 
holomorphic function
\begin{equation}\label{dnb}
 \lambda\mapsto D(\lambda) - B 
\end{equation}
from $\rho(A_D)$ into $\cL(H^{1/2}(\Gamma),H^{-1/2}(\Gamma))$.
Then the following holds.
\begin{tabel}
\item \label{main-1}
Let $\{f_0,\dots,f_k\}$ be a Jordan chain for $A_B$ at $\lambda_0$.
For all $m \in \{ 0,\ldots,k \} $ define $\varphi_m=\Tr f_m$.
Then $\{\varphi_0,\dots, \varphi_k\}$ is a Jordan chain
for the function \eqref{dnb} at $\lambda_0$.
\item \label{main-2}
Let $\{\varphi_0,\dots,\varphi_k\}$ be a Jordan chain for the function \eqref{dnb}
at $\lambda_0$.
Set $f_{-1} = 0$.
For all $m \in \{ 0,\ldots,k \} $ let 
$f_m\in H^1(\Omega)$ be the unique solution of the boundary value problem
\[
(\cA-\lambda_0)f_m=f_{m-1},\qquad \Tr f_m=\varphi_m.
\]
Then $\{f_0,\dots,f_k\}$ is a Jordan chain for $A_B$ at $\lambda_0$.
\end{tabel}
\end{theorem}

For the special case $k=0$ one obtains the following well-known result.

\begin{corollary}\label{maincor}
Adopt the notation and assumptions as in Theorem~\ref{main}.
Then the following holds.
\begin{tabel}
\item \label{mainc-1}
If $f_0$ is an eigenvector of $A_B$ at $\lambda_0$, then 
$D(\lambda_0)\Tr f_0=B\Tr f_0$ and $\Tr f_0\not=0$.
\item \label{mainc-2}
If $D(\lambda_0)\varphi_0=B\varphi_0$ and $\varphi_0\not=0$, 
then the unique solution $f_0\in H^1(\Omega)$ of the boundary value problem
\begin{equation*}
  (\cA-\lambda_0)f_0=0,\qquad \Tr f_0=\varphi_0,
\end{equation*}
is an eigenvector of $A_B$ at $\lambda_0$.
\end{tabel}
\end{corollary}

\begin{corollary} \label{cjordanchain410}
Adopt the notation and assumptions as in Theorem~\ref{main}.
Then 
\[
\Tr(\ker (A_B - \lambda_0)) = \ker( D(\lambda_0) - B)
\] 
and $\Tr$ is a bijection from $\ker (A_B - \lambda_0)$ onto $\ker( D(\lambda_0) - B)$.
\end{corollary}

\begin{remark}
We can mention here that the assumption $\lambda_0\in\rho(A_D)$ in Theorem~\ref{main} 
and Corollary~\ref{maincor} is really needed.
In fact, one may define the Dirichlet-to-Neumann graph as a linear relation 
consisting of the Cauchy data for all $\lambda_0\in\sigma_p(A_D)$.
By \cite{Fil1} Theorem~1 there exist $\mu > 0$, $\lambda \in \Ri$, 
$u \in C_c^\infty(\Ri^3) \setminus \{ 0 \} $ and a H\"older continuous function
$g \colon \Ri^3 \to [\mu,\infty)$ such that $- \divv g \nabla u = \lambda u$.
Let $\Omega$ be a Lipschitz domain with $\supp u \subset \Omega$.
Choose $c_{kl} = g|_\Omega \, \delta_{kl}$, $b_k = c_k = c_0 = 0$
for all $k,l \in \{ 1,\ldots,d \} $ and $f_0 = u|_\Omega$.
Let $B \in \cL(L_2(\Gamma))$.
Then $f_0$ is an eigenfunction of $A_B$ at $\lambda$.
But $\Tr f_0 = 0$.
So one cannot drop the assumption $\lambda_0 \in \rho(A_D)$ in 
Corollary~\ref{maincor}\ref{mainc-1}.
\end{remark}

Observe that the homogenenous and inhomogeneous boundary value problems in 
Theorem~\ref{main}\ref{main-2} and Corollary~\ref{maincor}\ref{mainc-2} 
admit unique solutions by Lemma~\ref{sollem}.
The proof of Theorem~\ref{main} requires quite some preparation.
The next lemma is particularly useful;
its proof is partly based on an argument that was given by V.A. Derkach for
symmetric and selfadjoint linear relations in Krein spaces; see also \cite{DerkachM3} Section 7.4.4.

\begin{lemma}\label{mainlem}
Let $A_B$ be the Robin realisation of $\cA$ in $L_2(\Omega)$ as 
in Proposition~\ref{ljordan2-313} and
let $\{f_0,\ldots,f_k\}$ be a Jordan chain of $A_B$ at $\lambda_0\in\rho(A_D)$.
For all $m \in \{ 0,\ldots,k \} $ define  
$\varphi_m=\Tr f_m \in H^{1/2}(\Gamma)$.
Let $\varphi\in H^{1/2}(\Gamma)$ and
let $g \in H^1(\Omega)$ be the unique solution of the adjoint problem
$(\widetilde \cA - \overline{\lambda_0}) g = 0$ such that 
$\Tr g =\varphi$. 
Then the following holds.
\begin{tabel}
\item \label{lem-1} 
If $j \in \{ 1,\ldots,k \} $, then
\begin{equation} \label{useful}
(f_{j-1},g)_{L_2(\Omega)}
=\langle D(\lambda_0)\varphi_j- B\varphi_j,\varphi\rangle_{H^{-1/2}(\Gamma) \times H^{1/2}(\Gamma)}.
\end{equation}
\item \label{lem-2}
If $j \in \{ 1,\ldots, k+1 \} $, then
\begin{equation} \label{iieq}
(f_{j-1},g)_{L_2(\Omega)}
= -\sum_{l=1}^j \frac{1}{l!} 
\langle D^{(l)}(\lambda_0)\varphi_{j-l},\varphi\rangle_{H^{-1/2}(\Gamma) \times H^{1/2}(\Gamma)}.
\end{equation}
\end{tabel}
\end{lemma}

\begin{proof}
For all $\lambda \in \rho(A_D)$ let 
$g_{\overline \lambda}\in H^1(\Omega)$ be the unique solution of the adjoint problem
$(\widetilde \cA - \overline \lambda) g_{\overline \lambda} = 0$ such that 
$\Tr g_{\overline \lambda}=\varphi$; see Lemma~\ref{sollem}\ref{sollem-1}.
Then $g_{\overline{\lambda_0}} = g$.
We set $f_{-1} = 0$.

`\ref{lem-1}'.
If $j \in \{ 0,\ldots,k \} $ and $\lambda \in \rho(A_D)$, then
$f_j \in \dom A_B$,
so $A_B f_j = \cA f_j$ and $\gamma_N f_j = B \Tr f_j$ by Proposition~\ref{ljordan2-313}.
Therefore 
\begin{eqnarray}
\lefteqn{
(A_B f_j,g_{\overline \lambda})_{L_2(\Omega)}
   - (f_j,\overline \lambda g_{\overline \lambda})_{L_2(\Omega)}
} \hspace*{30mm} \nonumber \\*
& = & (\cA f_j,g_{\overline \lambda})_{L_2(\Omega)}
   - (f_j,\widetilde\cA g_{\overline \lambda})_{L_2(\Omega)}  \nonumber\\
& = & \langle \Tr f_j,\widetilde\gamma_N g_{\overline \lambda}\rangle_{H^{1/2}(\Gamma) \times H^{-1/2}(\Gamma)}
   - \langle\gamma_N f_j,\Tr g_{\overline \lambda}\rangle_{H^{-1/2}(\Gamma) \times H^{1/2}(\Gamma)}  \nonumber\\
& = &\langle\Tr f_j,\widetilde D(\overline \lambda)\Tr g_{\overline \lambda}\rangle_{H^{1/2}(\Gamma) \times H^{-1/2}(\Gamma)}
   - \langle B \Tr f_j,\Tr g_{\overline \lambda}\rangle_{H^{-1/2}(\Gamma) \times H^{1/2}(\Gamma)}  \nonumber\\
& = & \bigl\langle D(\lambda) \varphi_j - B \varphi_j,\varphi\bigr\rangle_{H^{-1/2}(\Gamma) \times H^{1/2}(\Gamma)}
,
\label{elmainlem;1}
\end{eqnarray}
where we used (\ref{eGreenSecond}) in the last step.
Choosing $\lambda = \lambda_0$ gives
\begin{eqnarray*}
(f_{j-1},g)_{L_2(\Omega)}
& = & \bigl((A_B-\lambda_0)f_j,g_{\overline{\lambda_0}}\bigr)_{L_2(\Omega)}  \\
& = & (A_B f_j,g_{\overline{\lambda_0}})_{L_2(\Omega)}
   - (f_j,\overline{\lambda_0} g_{\overline{\lambda_0}})_{L_2(\Omega)}
= \bigl\langle D(\lambda_0) \varphi_j - B \varphi_j,\varphi\bigr\rangle_{H^{-1/2}(\Gamma) \times H^{1/2}(\Gamma)}
,   
\end{eqnarray*}
which proves (\ref{useful}).
Note that $j = 0$ gives 
$\langle D(\lambda_0) \varphi_0 - B \varphi_0,\varphi\rangle_{H^{-1/2}(\Gamma) \times H^{1/2}(\Gamma)} = 0$
and hence
\begin{equation}\label{neweq}
D(\lambda_0) \varphi_0 = B \varphi_0.
\end{equation}

`\ref{lem-2}'.
We shall show that  
\begin{equation}\label{formula}
- (f_{j-1},g_{\overline \lambda})_{L_2(\Omega)}
=\sum_{l=1}^j\left\langle\frac{1}{(\lambda-\lambda_0)^l}
\left(D(\lambda) - \sum_{s=0}^{l-1}\frac{1}{s!} (\lambda-\lambda_0)^s D^{(s)}(\lambda_0) \right)
\varphi_{j-l}, \varphi\right\rangle_{H^{-1/2}(\Gamma) \times H^{1/2}(\Gamma)}
\end{equation}
for all $j \in \{ 1,\ldots, k+1 \} $ and $\lambda \in \rho(A_D) \setminus \{ \lambda_0 \} $.
Once we have shown this, then the equality \eqref{iieq} easily follows by taking the 
limit $\lambda \to \lambda_0$.
In fact, the left hand side of \eqref{formula} tends to 
$-(f_{j-1},g_{\overline \lambda_0})_{L_2(\Omega)} = -(f_{j-1},g)_{L_2(\Omega)}$ by 
Lemma~\ref{ljordanchain201}\ref{ljordanchain201-1},
and using the Taylor expansion 
\[
D(\lambda)
=\sum_{s=0}^\infty \frac{1}{s!} (\lambda-\lambda_0)^s D^{(s)}(\lambda_0)
\]
it is easy to see that for $\lambda \to \lambda_0$ the right hand side in \eqref{formula} tends to 
\[
\sum_{l=1}^j \frac{1}{l!} \langle D^{(l)}(\lambda_0)\varphi_{j-l},\varphi\rangle_{H^{-1/2}(\Gamma) \times H^{1/2}(\Gamma)}.
\]

We prove formula \eqref{formula} by induction.
If $j=1$ and $\lambda \in \rho(A_D) \setminus \{ \lambda_0 \} $,
then (\ref{elmainlem;1}) gives
\begin{eqnarray*}
- (\lambda-\lambda_0)_{L_2(\Omega)} \, (f_0,g_{\overline \lambda})_{L_2(\Omega)}
& = & (\lambda_0 f_0,g_{\overline \lambda})_{L_2(\Omega)}
     -(f_0,\overline \lambda g_{\overline \lambda})_{L_2(\Omega)}  \\
& = & (A_B f_0,g_{\overline \lambda})_{L_2(\Omega)}
      -(f_0,\overline \lambda g_{\overline \lambda})_{L_2(\Omega)}  \\
& = & \bigl\langle D(\lambda) \varphi_0 - B \varphi_0,\varphi\bigr\rangle_{H^{-1/2}(\Gamma) \times H^{1/2}(\Gamma)} \\
& = & \bigl\langle (D(\lambda)-D(\lambda_0))\varphi_0,\varphi\bigr\rangle_{H^{-1/2}(\Gamma) \times H^{1/2}(\Gamma)}
,
\end{eqnarray*}
where we used \eqref{neweq} in the last step.
So \eqref{formula} is valid if $j = 1$.

Let $m \in \{ 1,\ldots,k \} $ and suppose that \eqref{formula} is valid for $j=m$.
Then by taking the limit $\lambda \to \lambda_0$ one deduces that 
\[
- (f_{m-1},g)_{L_2(\Omega)}
= \sum_{l=1}^m \frac{1}{l!} \langle D^{(l)}(\lambda_0)\varphi_{m-l},\varphi\rangle_{H^{-1/2}(\Gamma) \times H^{1/2}(\Gamma)},
\]
and together with \eqref{useful} we conclude
\begin{equation} \label{veryuseful}
 \bigl\langle D(\lambda_0) \varphi_m - B \varphi_m,\varphi\bigr\rangle_{H^{-1/2}(\Gamma) \times H^{1/2}(\Gamma)}
=-\sum_{l=1}^m\frac{1}{l!}\langle D^{(l)}(\lambda_0)\varphi_{m-l},
       \varphi\rangle_{H^{-1/2}(\Gamma) \times H^{1/2}(\Gamma)}
 .
\end{equation}
Now let us prove the formula \eqref{formula} for $j=m+1$.
Let $\lambda \in \rho(A_D) \setminus \{ \lambda_0 \}$.
Then a simple computation shows
\begin{eqnarray*}
\lefteqn{
\sum_{l=1}^{m+1}\frac{1}{(\lambda-\lambda_0)^l} 
  \left(D(\lambda) - \sum_{s=0}^{l-1}\frac{1}{s!} (\lambda-\lambda_0)^s D^{(s)}(\lambda_0) \right)
     \varphi_{m+1-l} 
} \hspace*{20mm}  \\*
& = & \sum_{l=2}^{m+1}\frac{1}{(\lambda-\lambda_0)^l}
    \left(D(\lambda) - \sum_{s=0}^{l-1}\frac{1}{s!} (\lambda-\lambda_0)^s D^{(s)}(\lambda_0) \right)
  \varphi_{m+1-l}\\*
& & \hspace*{10mm} {} 
    + \frac{D(\lambda)-D(\lambda_0)}{\lambda-\lambda_0} \varphi_m  \\ 
& = & \sum_{l=1}^m\frac{1}{(\lambda-\lambda_0)^{l+1}}
 \left(D(\lambda) - \sum_{s=0}^{l}\frac{1}{s!}
     (\lambda-\lambda_0)^s D^{(s)}(\lambda_0) \right) \varphi_{m-l}  \\*
& & \hspace*{10mm} {} 
   + \frac{D(\lambda)-D(\lambda_0)}{\lambda-\lambda_0}\varphi_m  \\
& = &  \frac{1}{\lambda-\lambda_0}\sum_{l=1}^m\frac{1}{(\lambda-\lambda_0)^{l}}
 \left(D(\lambda) - \sum_{s=0}^{l-1}\frac{1}{s!}
     (\lambda-\lambda_0)^s D^{(s)}(\lambda_0) \right) \varphi_{m-l}\\*
& & \hspace*{10mm} {} 
    - \frac{1}{\lambda-\lambda_0}\sum_{l=1}^m
  \frac{1}{l!}D^{(l)}(\lambda_0)\varphi_{m-l}
    + \frac{D(\lambda)-D(\lambda_0)}{\lambda-\lambda_0}\varphi_m
\end{eqnarray*}
and using \eqref{formula} for $j=m$ for the first term on the right hand side, 
and \eqref{veryuseful} for the second term on the right hand side
gives
\begin{eqnarray*}
\lefteqn{
\sum_{l=1}^{m+1}\left\langle\frac{1}{(\lambda-\lambda_0)^l} 
  \left(D(\lambda) - \sum_{s=0}^{l-1}\frac{1}{s!} (\lambda-\lambda_0)^s D^{(s)}(\lambda_0) \right)
     \varphi_{m+1-l},\varphi\right\rangle_{H^{-1/2}(\Gamma) \times H^{1/2}(\Gamma)} 
} \hspace*{20mm} \\*
& = & - \frac{1}{\lambda-\lambda_0}(f_{m-1},g_{\overline \lambda})_{L_2(\Omega)}
   +\frac{1}{\lambda-\lambda_0}
        \bigl\langle D(\lambda_0) \varphi_m - B \varphi_m,\varphi\bigr\rangle_{H^{-1/2}(\Gamma) \times H^{1/2}(\Gamma)} \\*
& & \hspace*{10mm} {}
   + \frac{1}{\lambda-\lambda_0}
  \bigl\langle D(\lambda) \varphi_m - D(\lambda_0) \varphi_m,\varphi\bigr\rangle_{H^{-1/2}(\Gamma) \times H^{1/2}(\Gamma)}\\*
& = & \frac{1}{\lambda-\lambda_0}
    \bigl\langle D(\lambda) \varphi_m - B \varphi_m,\varphi\bigr\rangle_{H^{-1/2}(\Gamma) \times H^{1/2}(\Gamma)}
  - \frac{1}{\lambda-\lambda_0}(f_{m-1},g_{\overline \lambda})_{L_2(\Omega)}\\
& = & \frac{1}{\lambda-\lambda_0}
    \bigl((A_B f_m,g_{\overline \lambda})_{L_2(\Omega)}
          - (f_m,\overline \lambda g_{\overline \lambda})_{L_2(\Omega)}
          - (f_{m-1},g_{\overline \lambda})\bigr)_{L_2(\Omega)}  \\
& = & \frac{1}{\lambda-\lambda_0}
    \bigl(( f_{m-1} + \lambda_0 \, f_m ,g_{\overline \lambda})_{L_2(\Omega)}
          - (f_m,\overline \lambda g_{\overline \lambda})_{L_2(\Omega)}
     -(f_{m-1},g_{\overline \lambda})\bigr)_{L_2(\Omega)}  \\
& = & -(f_m, g_{\overline \lambda})_{L_2(\Omega)},
\end{eqnarray*}
where \eqref{elmainlem;1} was used for $j=m$ in third equality and 
$(A_B-\lambda_0) f_m=f_{m-1}$ was used in the fourth equality.
We have shown \eqref{formula} for $j=m+1$.
The proof of \ref{lem-2} is complete.
\end{proof}

Now we are able to prove the main theorem.

\begin{proof}[{\bf Proof of Theorem~\ref{main}}]
`\ref{main-1}'. 
Let $\{f_0,\ldots,f_k\}$ form a Jordan chain for $A_B$ at $\lambda_0\in\rho(A_D)$ and let 
$\varphi_j = \Tr f_j\in H^{1/2}(\Gamma)$ for all $j\in\{1,\ldots,k\}$
be the corresponding traces.
We have to prove that
\begin{equation}\label{goal}
 \sum_{l=0}^j\frac{1}{l!} D^{(l)}(\lambda_0)\varphi_{j-l}=B\varphi_j
\end{equation}
for all $j\in\{0,\ldots,k\}$ and that $\varphi_0\not=0$.

Using Proposition~\ref{ljordan2-313} it is easy to see that 
\begin{equation*}
 D(\lambda_0)\varphi_0-B\varphi_0=D(\lambda_0)\Tr f_0-B\Tr f_0=\gamma_N f_0-\gamma_N f_0=0
\end{equation*}
and hence \eqref{goal} is valid if $j=0$.
Furthermore, 
$\varphi_0=\Tr f_0\not=0$ as otherwise 
$f_0\in\dom A_D$ and therefore
$(A_D-\lambda_0)f_0=(A_B-\lambda_0)f_0=0$,
which together with $\lambda_0\in\rho(A_D)$ would imply $f_0=0$.

Let $j \in \{ 1,\ldots,k \} $ and let $\varphi \in H^{1/2}(\Gamma)$.
Then Lemma~\ref{mainlem} gives
\[
\bigl\langle D(\lambda_0)\varphi_j-B\varphi_j,
         \varphi\bigr\rangle_{H^{-1/2}(\Gamma) \times H^{1/2}(\Gamma)}
= -\sum_{l=1}^j \frac{1}{l!} 
   \langle D^{(l)}(\lambda_0)\varphi_{j-l},\varphi\rangle_{H^{-1/2}(\Gamma) \times H^{1/2}(\Gamma)} .
\]
This implies that 
\begin{equation*}
 B\varphi_j=D(\lambda_0)\varphi_j+\sum_{l=1}^j\frac{1}{l!}D^{(l)}(\lambda_0)\varphi_{j-l}
 =\sum_{l=0}^j\frac{1}{l!}D^{(l)}(\lambda_0)\varphi_{j-l}
\end{equation*}
as required.

`\ref{main-2}'.
Assume that $\{\varphi_0,\ldots,\varphi_k\}$ 
form a Jordan chain of the function 
$\lambda\mapsto D(\lambda)-B$ at $\lambda_0$, that is, \eqref{goal}
is valid for all $j \in \{ 0,1,\ldots k \} $ and $\varphi_0\not=0$.
In the following we construct a Jordan chain $\{f_0,\ldots ,f_k\}$ of $A_B$ at 
$\lambda_0$ such that the corresponding
 traces are given by the set of vectors $\{\varphi_0,\ldots,\varphi_k\}$.
We proceed by induction.
According to Lemma~\ref{sollem}\ref{sollem-1} there exists a unique $f_0 \in H^1(\Omega)$ 
such that $(\cA - \lambda_0) f_0 = 0$ and 
 $\Tr f_0=\varphi_0$.
Making use of \eqref{goal} for $j=0$
 we obtain
 \begin{equation*}
  \gamma_N f_0
=D(\lambda_0)\Tr f_0
=D(\lambda_0)\varphi_0
= B\varphi_0=B\Tr f_0
 \end{equation*}
and hence $f_0\in\dom A_B $ with $(A_B-\lambda_0)f_0=0$ by
Proposition~\ref{ljordan2-313}.
Since $\varphi_0\not=0$ it is clear that also $f_0\not=0$.

Now let $m \in \{ 1,\ldots,k \} $ and assume that there are 
$f_0,\ldots,f_{m-1} \in H^1(\Omega)$ such that 
$\varphi_j = \Tr f_j$ for all $j \in \{ 0,\ldots,m-1 \} $ and 
the vectors $\{f_0,\ldots,f_{m-1}\}$ form a Jordan chain for $A_B$ at $\lambda_0$.
By Lemma \ref{sollem}\ref{sollem-2} 
there exists a unique vector $f_m \in H^1(\Omega)$ 
such that 
\begin{equation}\label{solbvp}
(\cA-\lambda_0)f_m = f_{m-1} \quad\text{and}\quad \Tr f_m = \varphi_m.
\end{equation}
We shall prove that $\gamma_N f_m = B \Tr f_m$.
Once we proved that, it follows that $f_m \in \dom A_B$ and 
$(A_B - \lambda_0) f_m = f_{m-1}$.

By assumption and \eqref{solbvp} one deduces that
\[
D(\lambda_0) \Tr f_m=
D(\lambda_0) \varphi_m 
= B \varphi_m - \sum_{l=1}^m\frac{1}{l!} D^{(l)}(\lambda_0)\varphi_{m-l}= B \Tr f_m - \sum_{l=1}^m\frac{1}{l!} D^{(l)}(\lambda_0)\varphi_{m-l}
.  \]
Let $\varphi\in H^{1/2}(\Gamma)$.
By Lemma~\ref{sollem}\ref{sollem-1} there exists a unique
$g\in H^1(\Omega)$ such that
$(\widetilde \cA - \overline{\lambda_0}) g = 0$ and
$\Tr g =\varphi$. 
Then
\begin{eqnarray*}
((\cA-\lambda_0)f_m,g)_{L_2(\Omega)}
&=&(\cA f_m,g)_{L_2(\Omega)}
   -(f_m,\overline \lambda_0 g)_{L_2(\Omega)}\\
&=&(\cA f_m,g)_{L_2(\Omega)} - (f_m,\widetilde\cA g)_{L_2(\Omega)}  \\
&=&\langle\Tr f_m,\widetilde\gamma_N g\rangle_{H^{1/2}(\Gamma) \times H^{-1/2}(\Gamma)}
   -\langle\gamma_N f_m,\Tr g\rangle_{H^{-1/2}(\Gamma) \times H^{1/2}(\Gamma)} \\
&=&\langle\Tr f_m,\widetilde D(\overline \lambda_0)\Tr g\rangle_{H^{1/2}(\Gamma) \times H^{-1/2}(\Gamma)}
   - \langle\gamma_N f_m,\Tr g\rangle_{H^{-1/2}(\Gamma) \times H^{1/2}(\Gamma)}  \\
&=&\bigl\langle D(\lambda_0) \Tr f_m -\gamma_N f_m,
             \varphi\bigr\rangle_{H^{-1/2}(\Gamma) \times H^{1/2}(\Gamma)}\\
&=&\left\langle B \Tr f_m  -\gamma_N f_m 
     - \sum_{l=1}^m\frac{1}{l!} D^{(l)}(\lambda_0)\varphi_{m-l},\varphi\right\rangle_{H^{-1/2}(\Gamma) \times H^{1/2}(\Gamma)}
.
\end{eqnarray*}
On the other hand,
as $\{f_0,\ldots,f_{m-1}\}$ is a Jordan chain of $A_B$ at $\lambda_0$ we have
\begin{eqnarray*}
 ((\cA-\lambda_0)f_m,g)_{L_2(\Omega)}
= (f_{m-1},g)_{L_2(\Omega)}
& = & - \sum_{l=1}^m\frac{1}{l!} 
     \langle D^{(l)}(\lambda_0)\varphi_{m-l},\varphi\rangle_{H^{-1/2}(\Gamma) \times H^{1/2}(\Gamma)}
\end{eqnarray*}
by Lemma~\ref{mainlem}\ref{lem-2}.
Therefore 
$\langle B \Tr f_m  -\gamma_N f_m, \varphi\rangle_{H^{-1/2}(\Gamma) \times H^{1/2}(\Gamma)} = 0$ 
for all $\varphi\in H^{1/2}(\Gamma)$.
Thus $\gamma_N f_m = B \Tr f_m$
as required.
So $ \{ f_0,\ldots,f_m \} $ is a Jordan chain for $A_B$ at $\lambda_0$ 
with traces $ \{ \varphi_0,\ldots,\varphi_m \} $.
\end{proof}

\begin{remark}
In the abstract setting of boundary triplets and their Weyl functions for 
adjoint pairs \cite{LyantzeStorozh, MalamudMog, Vainerman}
it is known under a natural unique continuation hypothesis that
the poles of the Weyl function correspond to the isolated eigenvalues of the 
fixed extension, see \cite[Theorem 4.4]{BrownMarlettaNW}.
See also \cite{BrownMarlettaNW2, BrownHMNW, BorogovacLuger}
for related results in the context of indefinite inner product spaces. 
\end{remark}

\section{Variations} \label{Sjordanchain5}

In the previous section we considered the Dirichlet-to-Neumann operator
$D(\lambda) \colon H^{1/2}(\Gamma) \to H^{-1/2}(\Gamma)$ and the 
Jordan chain with respect to the holomorphic operator function 
$\lambda \mapsto D(\lambda) - B$ from $\rho(A_D)$ into 
$\cL(H^{1/2}(\Gamma),H^{-1/2}(\Gamma))$, where $B\in\cL(H^{1/2}(\Gamma),H^{-1/2}(\Gamma))$ 
satisfies \eqref{bsemi}.

Except from the obvious ellipticity condition and to have a Lipschitz domain,
there were no conditions on the coefficients: merely bounded measurable and complex valued.

There are two other Dirichlet-to-Neumann operators that we consider in this section.

\subsection{$C^2$-domains} \label{Subjordanchain5.1}

Throughout this subsection we suppose that $\Omega$ is a $C^2$-domain, 
$c_{kl} \in C^1(\overline \Omega)$ and $b_k = 0$ for all $k,l \in \{ 1,\ldots,d \} $.
We summarise some regularity results that we need in this subsection.

\begin{lemma} \label{ljordanchain520}
\mbox{}
\begin{tabel} 
\item \label{ljordanchain520-1}
If $f \in H^2(\Omega)$, then $\Tr f \in H^{3/2}(\Gamma)$,
$\cA f \in L_2(\Omega)$ and 
\[
\gamma_N f = \sum_{k,l=1}^d \nu_k \Tr( c_{kl} \, \partial_l f)
\in H^{1/2}(\Gamma)
.  \]
Moreover, the map $f \mapsto \gamma_N f$ is continuous from 
$H^2(\Omega)$ into $H^{1/2}(\Gamma)$.
\item \label{ljordanchain520-2}
Let $\lambda \in \rho(A_D)$.
For all $\varphi \in H^{3/2}(\Gamma)$ there exists a unique 
$f \in H^2(\Omega)$ such that $(\cA - \lambda) f = 0$ and $\Tr f = \varphi$.
Moreover, the map $\varphi \mapsto f$ is continuous from $H^{3/2}(\Gamma)$
into $H^2(\Omega)$.
\item \label{ljordanchain520-3}
Let $\lambda \in \rho(A_D)$.
For all $h \in L_2(\Omega)$ and $\varphi \in H^{3/2}(\Gamma)$ there exists a unique 
$f \in H^2(\Omega)$ such that $(\cA - \lambda) f = h$ and $\Tr f = \varphi$.
\end{tabel}
\end{lemma}
\begin{proof}
`\ref{ljordanchain520-1}'.
This follows from \cite{Gris} Theorem~1.5.1.2 and the divergence theorem.

`\ref{ljordanchain520-2}'.
By \cite{Gris} Theorem~1.5.1.2 there exists an $f_0 \in H^2(\Omega)$ such that 
$\Tr f_0 = \varphi$.
Then it follows \cite{Evans} Theorem~6.3.4 that there exists a unique
$h \in H^2(\Omega)$ such that $(\cA - \lambda) h = (\cA - \lambda) f_0$
and $\Tr h = 0$.
Therefore $f = f_0 - h$ satisfies the requirements.
The uniqueness is easy.
The continuity follows from Lemma~\ref{sollem}\ref{sollem-1} and the 
closed graph theorem. 

`\ref{ljordanchain520-3}'.
This can be proved similarly.
\end{proof}

For all $\lambda \in \rho(A_D)$ define the Dirichlet-to-Neumann operator
$\widehat D(\lambda) \colon H^{3/2}(\Gamma) \to H^{1/2}(\Gamma)$ as follows.
Let $\varphi \in H^{3/2}(\Gamma)$.
By Lemma~\ref{ljordanchain520}\ref{ljordanchain520-2} there exists a unique $f \in H^2(\Omega)$
such that $(\cA - \lambda) f = 0$ and $\Tr f = \varphi$.
Define $\widehat D(\lambda) \varphi = \gamma_N f \in H^{1/2}(\Gamma)$ by 
Lemma~\ref{ljordanchain520}\ref{ljordanchain520-1}.
Then $\widehat D(\lambda)$ is a bounded operator.

Next we consider holomorphy.

\begin{lemma} \label{ljordanchain521}
The map $\lambda \mapsto \widehat D(\lambda)$ from $\rho(A_D)$ 
into $\cL(H^{3/2}(\Gamma), H^{1/2}(\Gamma))$ is holomorphic.
\end{lemma} 
\begin{proof}
For all $\varphi \in H^{3/2}(\Gamma)$ and $\psi \in H^{1/2}(\Gamma)$
define $\alpha_{\varphi,\psi} \colon \cL(H^{3/2}(\Gamma), H^{1/2}(\Gamma)) \to \Ci$
by 
\[
\alpha_{\varphi,\psi}(F) = (F \varphi, \psi)_{L_2(\Gamma)}
.  \]
Then $\alpha_{\varphi,\psi} \in \cL(H^{3/2}(\Gamma), H^{1/2}(\Gamma))^*$.
Let $W = \spann \{ \alpha_{\varphi,\psi} : 
      \varphi \in H^{3/2}(\Gamma) \mbox{ and } \psi \in H^{1/2}(\Gamma) \} $.
Since $H^{1/2}(\Gamma)$ is dense in $L_2(\Gamma)$, it follows that the 
space $W$ is separating, that is, if $F \in \cL(H^{3/2}(\Gamma), H^{1/2}(\Gamma))$
with $\alpha(F) = 0$
for all $\alpha \in W$, then it follows that $F = 0$.
If $\varphi \in H^{3/2}(\Gamma)$ and $\psi \in H^{1/2}(\Gamma)$, then 
\[
\alpha_{\varphi,\psi}(\widehat D(\lambda))
= (\widehat D(\lambda) \varphi, \psi)_{L_2(\Gamma)}
= \langle D(\lambda) \varphi, \psi \rangle_{H^{-1/2}(\Gamma)\times H^{1/2}(\Gamma)}
\]
for all $\lambda \in \rho(A_D)$.
Hence the map $\lambda \mapsto \alpha_{\varphi,\psi}(\widehat D(\lambda))$ is 
holomorphic for all $\varphi \in H^{3/2}(\Gamma)$ and $\psi \in H^{1/2}(\Gamma)$ by Lemma~\ref{ljordanchain201}\ref{ljordanchain201-2}.
Consequently the map $\lambda \mapsto \widehat D(\lambda)$ from $\rho(A_D)$ 
into $\cL(H^{3/2}(\Gamma), H^{1/2}(\Gamma))$ is holomorphic by 
\cite{ABHN} Theorem~A.7.
\end{proof}

The alluded variation of Theorem~\ref{main} is as follows.

\begin{theorem} \label{tjordanchain522}
Let $B \in \cL(H^{3/2}(\Gamma), H^{1/2}(\Gamma))$ and suppose there 
exists an $\eta > 0$ such that 
\begin{equation*}
\RRe ( B \varphi,\varphi)_{L_2(\Gamma)}
\leq \eta \|\varphi\|_{L_2(\Gamma)}^2
\end{equation*}
for all $\varphi\in H^{3/2}(\Gamma)$.
Let $A_B$ be the Robin realisation of $\cA$ in $L_2(\Omega)$ as 
in Proposition~\ref{ljordan2-313}, let $\lambda_0\in\rho(A_D)$ and consider the 
holomorphic function
\begin{equation} \label{etjordanchain522;1}
 \lambda \mapsto \widehat D(\lambda) - B 
\end{equation}
from $\rho(A_D)$ into $\cL(H^{3/2}(\Gamma),H^{1/2}(\Gamma))$.
Then the following holds.
\begin{tabel}
\item \label{tjordanchain522-1}
Let $f_0,\ldots,f_k \in H^2(\Omega)$.
Suppose that $\{f_0,\dots,f_k\}$ is a Jordan chain for $A_B$ at $\lambda_0$.
For all $m \in \{ 0,\ldots,k \} $ define $\varphi_m=\Tr f_m$.
Then $\{\varphi_0,\dots, \varphi_k\}$ is a Jordan chain
for the function \eqref{etjordanchain522;1} at $\lambda_0$.
\item \label{tjordanchain522-2}
Let $\{\varphi_0,\dots,\varphi_k\}$ be a Jordan chain for the function \eqref{etjordanchain522;1}
at $\lambda_0$.
Set $f_{-1} = 0$.
For all $m \in \{ 0,\ldots,k \} $ let 
$f_m\in H^2(\Omega)$ be the unique solution of the boundary value problem
\[
(\cA-\lambda_0)f_m=f_{m-1}
\quad\text{and}\quad
\Tr f_m=\varphi_m.
\]
Then $\{f_0,\dots,f_k\}$ is a Jordan chain for $A_B$ at $\lambda_0$.
\end{tabel}
\end{theorem}

The proof is similar to the proof of Theorem~\ref{main}, with obvious changes.

\subsection{m-Sectorial operators} \label{Subjordanchain5.2}

Throughout this subsection we merely assume again that $\Omega$ 
is a Lipschitz domain, but we put conditions on the coefficients 
of the elliptic operator.
We assume that $c_{kl} = c_{lk} \in W^{1,\infty}(\Omega,\Ri)$ is real valued and 
$b_k = c_k = 0$ for all 
$k \in \{ 1,\ldots,d \} $.
We emphasise that $c_0$ can be complex valued and merely measurable.
An example is the Schr\"odinger operator with complex potential.
The Dirichlet-to-Neumann operator $D(\lambda) \colon H^{1/2}(\Gamma) \to H^{-1/2}(\Gamma)$
has been studied intensively in \cite{BeE1, BGHN, GesM2, GesM3}.
Let $\cD(\lambda)$ be the part of $D(\lambda)$ in $L_2(\Gamma)$.
So $\cD(\lambda) \subset D(\lambda)$ and if $\varphi \in L_2(\Gamma)$, 
then $\varphi \in \dom \cD(\lambda)$ if and only if 
$\varphi \in H^{1/2}(\Gamma)$ and $D(\lambda) \varphi \in L_2(\Gamma)$.
The operator $\cD(\lambda)$ can be represented by a form.

\begin{lemma} \label{ljordanchain530}
Let $\lambda \in \rho(A_D)$.
Let $\varphi,\psi \in L_2(\Gamma)$.
Then the following are equivalent.
\begin{tabeleq}
\item \label{ljordanchain530-1}
$\varphi \in \dom \cD(\lambda)$ and $\cD(\lambda) \varphi = \psi$.
\item \label{ljordanchain530-2}
There exists an $f \in H^1(\Omega)$ such that $\Tr f = \varphi$ and 
\[
\gota(f,g) - \lambda (f,g)_{L_2(\Omega)}
= (\psi,\Tr g)_{L_2(\Gamma)}
\]
for all $g \in H^1(\Omega)$.
\end{tabeleq}
\end{lemma}

The easy proof is left to the reader.

It seems that the domain of $\cD(\lambda)$ depends on $\lambda$.
This is not the case because of the restriction on the principal part
of the elliptic operator.
We collect the main properties of the operator $\cD(\lambda)$ in the 
next proposition.

\begin{proposition} \label{pjordanchain531}
\begin{tabel}
\item \label{pjordanchain531-1}
If $\lambda \in \rho(A_D)$, then the operator $\cD(\lambda)$ is m-sectorial.
\item \label{pjordanchain531-1.5}
If $\lambda \in \rho(A_D)$, then  $\dom \cD(\lambda) = H^1(\Omega)$.
\item \label{pjordanchain531-2}
The map $\lambda \mapsto \cD(\lambda)$ from $\rho(A_D)$ 
into $\cL(H^1(\Gamma), L_2(\Gamma))$ is holomorphic.
\end{tabel}
\end{proposition}
\begin{proof}
`\ref{pjordanchain531-1}'.
See \cite{Ouh8} Corollary~2.3.

`\ref{pjordanchain531-1.5}'.
`$\subset$'.
Let $\varphi \in \dom \cD(\lambda)$.
Then there exists an $f \in H^1(\Omega)$ such that 
$\varphi = \Tr f$ and $(\cA - \lambda)f = 0$.
So $\cA f = \lambda f \in L_2(\Omega)$ and 
$\gamma_N f = \cD(\lambda) \varphi \in L_2(\Gamma)$.
Therefore \cite{McL} Theorem~4.24(ii) implies that 
$\varphi = \Tr f \in H^1(\Gamma)$.

`$\supset$'.
Let $\varphi \in H^1(\Gamma)$.
By Lemma~\ref{sollem}\ref{sollem-1} there exists a unique $f \in H^1(\Omega)$ 
such that $(\cA - \lambda)f = 0$ and $\Tr f = \varphi$.
Then $\cA f = \lambda f \in L_2(\Omega)$.
Hence \cite{McL} Theorem~4.24(i) gives $\gamma_N f \in L_2(\Gamma)$.
So $\varphi \in \dom \cD(\lambda)$.

`\ref{pjordanchain531-2}'.
For all $\varphi \in H^1(\Gamma)$ and $\psi \in H^{1/2}(\Gamma)$
define $\alpha_{\varphi,\psi} \colon \cL(H^1(\Gamma), L_2(\Gamma)) \to \Ci$
by 
\[
\alpha_{\varphi,\psi}(F) = (F \varphi, \psi)_{L_2(\Gamma)}
.  \]
Then argue as in the proof of Lemma~\ref{ljordanchain521}.
\end{proof}

Now we are able to formulate another version of Theorem~\ref{main}.

\begin{theorem} \label{tjordanchain534}
Let $B \in \cL(H^1(\Gamma), L_2(\Gamma))$ and suppose that there exists an
$\eta > 0$ such that 
\begin{equation*}
\RRe ( B \varphi,\varphi)_{L_2(\Gamma)}
\leq \eta \|\varphi\|_{L_2(\Gamma)}^2
\end{equation*}
for all $\varphi\in H^1(\Gamma)$.
Let $A_B$ be the Robin realisation of $\cA$ in $L_2(\Omega)$ as 
in Proposition~\ref{ljordan2-313}, let $\lambda_0\in\rho(A_D)$ and consider the 
holomorphic function
\[
 \lambda \mapsto \widehat D(\lambda) - B 
\]
from $\rho(A_D)$ into $\cL(H^1(\Omega),L_2(\Gamma))$.
Then the following holds.
\begin{tabel}
\item \label{tjordanchain534-1}
Let $\{f_0,\dots,f_k\}$ be a Jordan chain for $A_B$ at $\lambda_0$.
For all $m \in \{ 0,\ldots,k \} $ define $\varphi_m=\Tr f_m$.
Then $\{\varphi_0,\dots, \varphi_k\}$ is a Jordan chain
for the function \eqref{etjordanchain522;1} at $\lambda_0$.
\item \label{tjordanchain534-2}
Let $\{\varphi_0,\dots,\varphi_k\}$ be a Jordan chain for the function \eqref{etjordanchain522;1}
at $\lambda_0$.
Set $f_{-1} = 0$.
For all $m \in \{ 0,\ldots,k \} $ let 
$f_m\in H^1(\Omega)$ be the unique solution of the boundary value problem
\[
(\cA-\lambda_0)f_m=f_{m-1}
\quad\text{and}\quad
\Tr f_m=\varphi_m.
\]
Then $\{f_0,\dots,f_k\}$ is a Jordan chain for $A_B$ at $\lambda_0$.
\end{tabel}
\end{theorem}

The proof is similar to the proof of Theorem~\ref{main}, with obvious changes.

\end{document}